%% file: optdd2.tex
\documentclass[
a4paper,
twoside,
11pt
]{amsart}
\usepackage[pdftex,
colorlinks, 
bookmarksnumbered, 
bookmarksopen, 
linkcolor=blue, 
citecolor=blue, 
]{hyperref}
\usepackage{amscd}
\usepackage{eucal}
\usepackage{mathptmx}

\input{definitionen1}
\providecommand {\linprod}[1]{[ #1 ]}
\newcommand \dg {dyadic group}
\newcommand \dgn {\bz_2^\bn}
\newcommand \fdgn {\bz_2^n}
\newcommand \zii {\ensuremath{\bz_2}}
\DeclareMathOperator \bp {\dot{+}}
\DeclareMathOperator \dc {\dot{\ast}}
\DeclareMathOperator {\trace} {trace}
\newcommand \wf {Walsh function}

\newcommand \bd {\mathbb{D}}
\newcommand \hs {F}
\newcommand \mW {\ensuremath{\mathcal{W}}}
\newcommand \WT {Walsh transform}
\newcommand \qu {\ensuremath{[0;1)}}
\newcommand {\mI} {\ensuremath{\mathcal{I}}}
\newcommand {\mH} {\ensuremath{\mathcal{H}}}
\newcommand {\mK} {\ensuremath{\mathcal{K}}}
\newcommand \HS {Hilbert-Schmidt}
\begin{document}
\title{The Optimal Dyadic Derivative}
\date{\today} 
\author{Andreas Klotz}
\address{Faculty of Mathematics\\University of Vienna \\
  Nordbergstrasse 15\\A-1090 Vienna, AUSTRIA} 
\email{andreas.klotz@univie.ac.at} \thanks{ A.~K.~ was supported by the FWF project P22746N13}
\begin{abstract}
 We show that the best approximation to the difference operators on the cyclic groups of order $2^n$ by a dyadic convolution operator are the restrictions of a generalized dyadic derivative. This  answers a question on the ``intuitive'' interpretation of the dyadic derivative posed by Butzer and Wagner more than 30 years ago.
\end{abstract}
\maketitle
\section{Introduction}
\label{sec:introduction}
\input{introv3}

\section{Background and Notation}
\label{sect:not}

The dyadic expansion of $x \in \qu$ is $x=\sum_{k=1}^ \infty x_k 2^{-k}$ with $x_k \in \set{0,1}$; if $x= 2^{-m}k$ for an integer $k$ we choose the expansion ending in zeros.
The dyadic expansion of $n \in \bn_0$ is $n=\sum_{k=0}^\infty n_k 2^k$ with $n_k \in \set{0,1}$.

Denote by  $\zii$ the cyclic group of order 2, i.e. $\zii =\set{0,1}$ with addition modulo 2. The Haar measure on $\zii$ assigns the value $1/2$ to each singleton.
The \dg\ $\dgn=\prod_{k=1}^\infty\zii $ consists of all sequences $(x_k)_{k \in \bn}$ with $x_k \in \zii$. It is generated by the elements  $e_k= (\delta_{jk})_{k \in \bn}$.  Addition $\bp$ in $\dgn$
 is defined componentwise modulo 2:    $(x \bp y)_k= x_k + y_k \mod 2$. 
 With the natural product topology  $\dgn$ is a totally disconnected, compact abelian group. 
The  Haar measure  ${{d}}\lambda$ on $\dgn$ is the product measure of the Haar measure on $\zii$ with total mass one.

For $x \in \dgn$ the dyadic balls  
\[
B_n(x)=x \bp \set{z \in \dgn \colon z_k=0, 1\leq k \leq n}
      =  \set{z \in \dgn \colon z_k=x_k, 1\leq k \leq n}
\]
 form a finite partition of $\dgn$ for each nonnegative integer $n$, and are a basis of the topology of $\dgn$. Dyadic balls are either disjoint or included in each other. The balls $B_n(0)$ are subgroups of $\dgn$.
 In particular $\dgn$ can be partitioned uniquely into the balls
$
 \set{B_n(x)}_{x \in \fdgn} 
$ for each $n \in \bn$.

The characters on $\dgn$, i.e. the group homomorphisms to the complex unit circle $\bt$ are the \emph{\wf s}, defined for  $x \in\dgn$ and for $n \in \bno$  by 
\[
w_n(x)=(-1)^{\sum_{k=0}^\infty n_k x_{k+1}} =(-1)^{\linprod{n,x}}\,,
\]
where ${\linprod{n,x}}=\sum_{k=0}^\infty n_k x_{k+1} \mod 2$. The enumeration is called the Paley enumeration, and $w_n$ are the Walsh-Paley functions.

The multiplicative group $\bd=\set{w_m \colon m \in \bno}$ is the Pontryagin dual of the \dg, and an orthonormal basis of $L^2(\dgn)$. If $m, n \in \bn_0$ with dyadic expansions as above, then dyadic addition 
$
m \bp n= \sum_{k=0}^\infty (m_k \bp n_k) 2^k
$
turns $\bno$ into a group which  is isomorphic to $\bd$.   
 In particular, for $x, y \in \dgn $, and $m, n \in \bno$
\begin{align*}
w_mw_n&=w_{m \bp n} \,, \\
w_m(x \bp y) &=w_m(x)w_m(y)\,.
\end{align*}
For $n \in \bno$ the sets 
$\bd_n = \set{w_m \in \bd \colon m < 2^n}$ are  multiplicative subgroups of $\bd$ of order $2^n$, and the functions in $\bd_n$ are constant on the dyadic intervals $B_n(x)$. In particular,
  the functions in $\bd_n$ are uniquely determined by their values on  $\fdgn$.
 In other words, the span of $\bd_n$ is
\[
\mF_n=\set{f \colon \dgn \to \bc , \quad f \; \text{ constant on cosets of } \; B_{n}} \,.
\]

\subsection{Sequency ordering}
Besides the Paley ordering, there is another ordering of the \wf s -- the \emph{sequency} ordering --
 that is more closely related to the frequency concept of Fourier analysis. To define the sequency ordering we need the interpretation of \wf s as being defined on the unit interval.

For $x \in \qu$ we define \emph{Fine's map} 
$\rho: \qu \to \dgn; \; x \mapsto (x_k)_{k \in \bn}$, where the $x_k$ are the components of the dyadic expansion of $x$.
The absolute value of $x \in \dgn$ is $\abs{x}=\sum_{k=1}^ \infty x_k2^{-k}$, a metric on $\dgn$ is given by $d(x,y)= \abs{x-y}$.
By abuse of notation we use the absolute value also for $k \in \bno$.

On the interval $\qu$  the \wf s are defined by
\[
\widetilde w_n(x)=w_n(\rho(x)) \,,
\]
and the set $\widetilde \bd = \set{\widetilde w_n\colon n \in \bno}$ with pointwise multiplication is isomorphic to $\bd$ as a group; the definition of the subgroups  $\widetilde \bd_n$ is obvious.
For $1 \leq p < \infty$ the Banach spaces $L^p(\qu)$  and $L^p(\dgn)$ are isometrically isomorphic via the mapping (see, e.g.,\cite{ScWaSi90})
\[
L^p(\dgn) \to L^p(\qu); \; f \mapsto f \circ \abs{\phantom{x}}
\]

   We want to order the \wf s in increasing  \emph{sequency}, that is the number of sign changes on $(0;1)$. 
We  develop the needed concepts from scratch, as we shall  use  them  later, see also \cite{ScWaSi90} for a similar derivation.
\begin{lem}\label{lem-sign-change}
 The function  $\widetilde{w_k} \in \widetilde{\bd_n}$  changes sign at $x \in (0;1)$  only if $\rho(x) \in \fdgn$. If $x=\sum_{j=1}^{M(x)-1}x_j2^{-j} +2^{-M(x)}$,
  where   $1 \leq M(x) \leq n$, set $h_{M(x)}=\sum_{k=M(x)}^N e_{k} $. Then $\widetilde{w_k} \in \bd_n$  changes sign at $x$ if and only if $\linprod{k, h_{M(x)}}=1$.
\end{lem}
\begin{proof}
As the values of $w_k \in \bd_n$ are constant on the balls ${B_n(x)}$,  the functions $\widetilde w_k$ can change sign only at the points of the form $\sum_{j=1}^nx_j2^{-j}$, $x_j \in \zii$. So   $\widetilde w_k$ changes  sign  at $x= \sum_{j=1}^nx_j2^{-j}$,  if
\begin{equation}
  \label{eq:2}
  \widetilde w_k(x)=-\widetilde w_k({x}-2^{-n}) \,.
\end{equation}
 Equivalently, for $x=\sum_{j=1}^nx_je_{j}$
\begin{equation}\label{eq:5}
  w_k(x \bp \rho( \abs{x-2^{-n}})=-1 \,.
\end{equation}
If  $x=\sum_{j=1}^{M(x)-1}x_je_j +e_{M(x)}$ for $1 \leq M(x) \leq n$, then
\[
 \rho( \abs{x-2^{-n}})=\sum_{j=1}^{M(x)-1}x_je_j+ \sum_{j=M(x)+1}^ne_j \,,
\]
so
\begin{equation}\label{eq:3}
  x \bp \rho( \abs{x-2^{-n}}) = \sum_{j=M(x)}^n e_j =h_{M(x)}\,.
\end{equation}
The statement of the Lemma follows by combining equations (\ref{eq:5}) and  (\ref{eq:3}).
\end{proof}
\begin{rem}
  We note for later use that, by setting $M(0)=1$ and interpreting the
  difference $x-2^{-n}$ modulo $1$,
  i.e. $-2^{-n}=\sum_{j=1}^n2^{-j}$, Equation~(\ref{eq:3}) is also
  true for $x=0$.
\end{rem}

Define the $\zii$-linear mapping $\mS$ for $n=\sum_{j=0}^\infty n_j 2^j \in \bno$  by
$
\mS n  = \sum_{j=0}^\infty n_{j+1} 2^j \,.
$ 
Set $\mG= \mI + \mS$, where \mI\ is the identity on \bno.
The matrix representation of $\mG k$ on the \zii-subspace $\set{0,1, \dotsc, 2^n-1}$ of \bno\ with respect to the basis $2^l, l=0,\dotsc,n-1$ is
 \begin{equation*}
  \mG k= 
  \begin{pmatrix}
   1 & 1 & 0 & \hdotsfor{2} & 0 \\
   0 & 1 &  1 & 0 &\dotsc  & 0  \\
   \hdotsfor{6} \\
   0 & \hdotsfor{2} & 0& 1 &  1 \\
   0 & \hdotsfor{3} & 0& 1  
  \end{pmatrix}
\begin{pmatrix}
k_{0} \\ k_{1} \\ \dotsc \\k_{n-2}\\ k_{n-1}
\end{pmatrix}
=
\begin{pmatrix}
  k_{0} \bp k_1 \\  k_{1} \bp k_{2}  \\ \dotsc \\k_{n-2} \bp k_{n-1}\\ k_{n-1}  
\end{pmatrix}
\end{equation*}
\begin{prop}\label{prop:sequency-ordering}
  The function $\widetilde w_{\mG k}$ has exactly $k$ sign-changes in  $(0;1)$.
\end{prop}
\begin{proof}
 Equation \eqref{eq:3} implies that whenever $w_k$ changes sign at $x \in \fdgn$  it changes sign for all $z \in \fdgn$ with $M(z)=M(x)$. Define 
$M_r=\set{x \in \fdgn \setminus \set{0} \colon M(x)=r}$ for $1 \leq r \leq n$. By definition of  $M(x)$ we obtain that $\abs{M_r}=2^{r-1}$, and the disjoint union of the $M_r$ is
$\bigcup_{r=1}^n M_r =\fdgn \setminus \set{0}$.
As $w_k$ changes sign at $x \in M_r$, if and only if
\begin{align*}
  1= \linprod{k, h_{r}}
    =\linprod{k, \sum_{i=r}^ne_i}
= \sum_{i=r-1}^{n-1} k_i \mod 2 \,,
\end{align*}
the total number of sign changes of $w_k$ in the unit interval is 
\begin{align*}
 & \sum_{r=1}^n \bigl( \sum_{i=r-1}^{n-1} k_i \mod 2 \bigr)2^{r-1}
=\sum_{r=0}^{n-1} \bigl( \sum_{i=r}^{n-1} k_i \mod 2 \bigr)2^{r}\\
&=
\begin{vmatrix}
  \begin{pmatrix}
  k_{n-1} \bp k_{n-2} \bp \dotsc \bp k_0\\  \dotsc \\  k_{n-1} \bp k_{n-2} \\  k_{n-1}  
\end{pmatrix}
\end{vmatrix} 
=
\begin{vmatrix}
  \begin{pmatrix}
    1 & 1 & \hdotsfor{2} & 1\\
    0 & 1 & \hdotsfor{2} & 1\\ 
    \hdotsfor{5} \\
      0 & 0 & \dotsc & 0 & 1
 \end{pmatrix}
  \begin{pmatrix}
    k_{0} \\ k_{1} \\ \dotsc \\ k_{n-1}
  \end{pmatrix}
\end{vmatrix} \\
&= \bigabs{\sum_{i=0}^{n-1} \mS^i k}
= \abs{\mG ^{-1}k} \,.
\end{align*}
%
So $w_{Gk}$ has $k$ sign changes in $(0;1)$.
\end{proof}
 \section{Best Approximation by Dyadic Convolution Operators}
\label{sec:sequency-ordering}

\subsection{\HS\ norm}
\label{sec:hs-norm}
Let $\mH$ be a Hilbert space with orthonormal basis $\set{e_\lambda}_{\lambda \in \Lambda}$. The Hilbert-Schmidt-norm (or Frobenius norm) of an operator  $A \in \mB(\mH)$  is
$
\norm{A}_{\hs}= \bigl(\sum_{\lambda \in \Lambda}\norm{A e_\lambda}_\mH^2 \bigr)^{1/2}
$, and this norm is independent of the choice of the basis. 
It is well known that this norm stems from an inner product. With the trace operator  $\trace{A}=\sum_{\lambda \in \Lambda} \inprod {A e_\lambda, e_\lambda}$ this inner product is
$
\inprod{A,B}_{\hs}= \trace(B^* A) \,.
$
Therefore the (unique) best approximation of a \HS\ operator $A$ on a (closed) subspace  of the \HS\ operators in the \HS-norm is  the orthogonal projection on this subspace with respect to the \HS\ inner product.

%
We will  need the following immediate consequence of the definition of a \HS\ operator.
\begin{prop}\label{prop:hs--norm-equivalence}
  If $A$ is a \HS\ operator on $\mH$,  and $T \colon \mH \to \mK$ a Hilbert space isomorphism, then $TA \inv T $ is a \HS\ operator on $\mK$ with the same norm. 
\end{prop}

\subsection{\WT\ and dyadic convolution operators}
\label{sec:wt-dyad-conv}
  Let $f,g \in L^1(\dgn)$. The dyadic convolution of $f$ and $g$ is
\[
f \dc g (x) =\int_{t \in \dgn} f(x \bp t) g(t) \dd \lambda (t) \,.
\]
The dyadic convolution operator $C_f$ is  $C_f g = f \dc g$.\\
  The  \emph{Walsh Transform} of $f$  is
  \begin{equation*}
    \mW f (k)= f^\sim(k) = \inprod{f,w_k}=\int_{t \in \dgn} f(t) w_k(t) \dd \lambda(t) \, \text{ for } k \in \bno \,.
  \end{equation*}
The \WT\ is an isometry between $L^2(\dgn)$ and $\ell^2(\bno)$. 

It is straightforward that for
 $f,g \in \mF_n$
 \begin{equation*}
 f^\sim (k) = 
\begin{cases}
2^{-n}\sum_{t \in \fdgn} f(t) w_k(t), \quad  & k < 2^n\\
0, \quad & \text{else} \,,
\end{cases}
\end{equation*}
and
\[
f(x)=\sum_{k=0}^{2^n-1} f^\sim(k)w_k(x) \,.
\]
Following an approach of Pearl~\cite{Pearl75} we want to characterize the best approximation of operators on $L^2(\dgn)$ by dyadic convolution operators $C_f$, $f \in L^2(\dgn)$. 
\begin{prop}\label{prop:bestappr-dyconvop}
  Assume that $A$ is a \HS -operator on $L^2(\dgn)$. Then the best approximation of $A$ by a dyadic convolution operator $C_f$ in the \HS-norm is given by $f^\sim(k)= \inprod{A w_k, w_k}$.
\end{prop}
\begin{proof}
Consider the commutative diagram
\[
\begin{CD}
L^2(\dgn) @>A >> L^2(\dgn)\\
@V \simeq V \mW V @V \simeq V \mW V\\
\ell^2(\bno) @> \tilde A >> \ell^2(\bno)
\end{CD}
\]
i.e. $\tilde A \tilde f = \mW (A f)$, and $\norm{A}_{\hs}=\norm{\tilde A}_{\hs}$.
For the entries of the matrix $\tilde A$ we obtain
\begin{align*}
\tilde A(k,l)=&\inprod{\tilde A e_l, e_k}_{\ell^2(\bno)} =\inprod{\tilde A \,\mW \, w_l,  \mW\, w_k}_{\ell^2(\bno)} \\
=&\inprod{\mW^* \tilde A \, \mW \, w_l,   w_k}_{L^2(\dgn)} = \inprod{ A  w_l,   w_k}_{L^2(\dgn)} \,.
\end{align*}
 As $\widetilde C_f =\diag (\tilde f)$, where $\diag$ denotes a diagonal matrix, we see that minimal norm $\norm{A-C_f}_{\hs}$ is obtained for $\tilde f(k) = \inprod{A w_k, w_k}$. We still have to show that $f \in L^2(\dgn)$:
\[
\norm{f}_2^2=\norm{\tilde f}_2^2=\sum_{k=0}^\infty \abs{\inprod{A \, w_k, w_k}}^2 \leq \sum_{k=0}^\infty \norm{A\,w_k}_2^2 =\norm{A}_{\hs}^2 \,. \qedhere
\] 
\end{proof}
  \subsection{Best Approximation of Cyclic Difference Operators}
  \label{sec:best-appr-cycl}
We will need an appropriate definition of a dyadic derivative. 
In \cite{BuWa73,BuWa75} Butzer and Wagner introduced the concept of a \emph{dyadic derivative}, which was an extension of the \emph{logical derivative} used by Gibbs \cite{gibbs1970} on $\fdgn$. This operator has the property that
\[
D w_k = \abs k w_k \text{  for all } k \in \bno \,,
\]
thus mimicking the behavior of the classical differentiation operator on the exponentials. 
In \cite{On79}, Onneweer mentioned the dependence of the definition on the ordering on the \wf s, and offered an alternative definition of a dyadic derivative with the property that
\[
D w_k = 2^{\floor{\log_2 \abs k}} w_k \,.
\] 
We follow the definition of He Zelin ~\cite{Zelin93}, who defined a generalized dyadic derivative. We adapt the definition for $\dgn$.

Assume that $\gamma= (\gamma(k))_{k \in \bno}$ is a sequence of complex numbers. The \emph{generalized dyadic derivative }   $f \in L^1(\dgn)$ is 
\begin{equation}
D_\gamma f = \lim_{n \to \infty} f \dc {\sum_{k=0}^{2^n-1}  \gamma (k) w_k} \,,\label{eq:6}
\end{equation}
whenever the expression on the right side converges in $L^1(\dgn)$. The set of all $f \in L^1(\dgn)$ such that the limit exists is the \emph{domain} $\mD(D_\gamma)$ of $D_\gamma$.
The generalized dyadic derivative is a closed operator from $\mD(D_\gamma)$ to $L^1(\dgn)$. This follows form \cite[Thm 1]{Zelin93}.
It is elementary that
\[
D_\gamma w_k = \gamma(k) w_k \,.
\]
The case $\gamma(k)=\abs{k}$ corresponds to the dyadic derivative considered by Butzer and Wagner, and $\gamma(k)=  2^{\floor{\log_2\abs  k}}$ is the dyadic derivative of Onneweer.
%

For the formulation of the main theorem we need to introduce the translation operator on $[0;1)$ as $T_xf (t)= f(t-x)$, the difference is to be understood modulo 1, and the difference operator $\Delta_n = 2^n (T_{2^{-n}} -I)$. By abuse of notation these operators are also defined on $\fdgn$, if   $x \in \fdgn$.

\begin{thm} \label{thm_main}
  There is a unique generalized dyadic derivative $D_\gamma$ that approximates the ordinary derivative in the following sense:

For any $n \in \bn$, the restriction of $D_\gamma$ is the best approximation to $\Delta_n$ in the class of dyadic convolution operator, measured in the \HS\ norm,
\begin{equation}
  \label{eq:7}
  \norm{D_\gamma\vert_{\mF_n} - \Delta_n}_{\hs} = \min \set{\norm{C_f-\Delta_n}_{\hs} \colon  f \in \mF_n} \,.
\end{equation}
The coefficients of $\gamma$ are
\[
\gamma(\mG k) = 2 (k_0 + \abs k)
\]
In particular,
\[
D_\gamma f = 2 \lim_{n \to \infty} f \dc \sum_{k=0}^{2^n-1} (k_0 +\abs k) w_{\mG k} \,.
\]
\end{thm}
\begin{rems}
  (a) It is remarkable that \emph{one} operator satisfies (\ref{eq:7}) for \emph{all} indices $n$. \\
  (b) The theorem answers in a way the question posed by Butzer and Wagner: There is a (generalized) dyadic derivative that can be uniquely described as the limit of optimal approximations to  classical  difference operators by dyadic convolution operators. So in this way the dyadic derivative is the best approximation to the classical differentiation operator. It should be no surprise that the description of  $D_\gamma$ is simplest for the sequence ordering of the \wf s, as sequency mimics frequency. (Another example would be the position of the maximum of the Fourier transform of \wf s, which can be described most transparently for the sequency ordering.)\\
(c) On an $n$-dimensional Hilbert space the \HS-norm admits  a statistical interpretation~\cite[Thm. 2.2]{boettcher03}: If $X$ is a uniformly distributed random variable on the unit sphere $S_{n-1}$ then
\[
   \norm{A}_{\hs}^2= n \mathbb{E}(\norm{AX}_2^2) 
\]
where  $\mathbb{E}$ is the expectation operator. (Remark: A simple proof of this statement can be based on the divergence theorem.)
So the generalized dyadic derivative obtained in the theorem approximates the classical difference operators best in a statistical sense. This might be of interest in signal processing applications. 
\end{rems}
The statements of theorem are a simple consequence of the following lemma.
\begin{lem}\label{lem-main}
  The best approximation of the cyclic translation operator $T_{2^{-n}}$ by a dyadic convolution operator $C_f$ on $\mF_n$ in the \HS-norm satisfies
\[  f^\sim (\mG m)= 1-2^{-n+1}(\abs m +m_0) \] 
\end{lem}
\begin{proof}
This follows from
\begin{align*}
 f^\sim (k) &= 2^{-n}\sum_{x \in \fdgn}  T_{2^{-n}}w_k(x)w_k(x)=
 2^{-n}\sum_{x \in \fdgn}  w_k(\phi(\abs x-2^{-n}) \bp x) \\
&= 2^{-n}\sum_{x \in \fdgn} w_k(h_M(x))=2^{-n}(w_k(h_M(0))+\sum_{x \in \fdgn\setminus \set{0}} w_k(h_M(x))
\end{align*}
by \eqref{eq:3}. As $\abs {M_k} = 2^{k-1}$ for $k>0$ and $h_m(0)=h_1$ the sum above can be rewritten as
\begin{align*}
   f^\sim (k) &=2^{-n}\bigl(w_k(h_1)+\sum_{r=1}^n w_k(h_r)2^{r-1}\bigr) \\
              & =2^{-n}\bigl((-1)^{\linprod{k,h_1}}+\sum_{r=1}^n (-1)^{\linprod{k,h_r}}2^{r-1}\bigr)
\end{align*}
We now use that $(-1)^m = 1 - 2 m$  for $m \in \zii$ to obtain
\begin{align*}
   f^\sim (k) &= 2^{-n}\bigl(1-2{\linprod{k,h_1}}+\sum_{r=1}^n (1- 2{\linprod{k,h_r}})2^{r-1}\bigr) \\
             &= 2^{-n}\bigl(-2{\linprod{k,h_1}}+ 2^n -2\sum_{r=1}^n {\linprod{k,h_r}}2^{r-1}\bigr) \\
             &= 2^{-n}\bigl(-2{\linprod{k,h_1}}+ 2^n -2\abs{\inv \mG k}\bigr) \,.
\end{align*}
Substituting $k=\mG m$ yields
\begin{equation*}
   f^\sim (\mG m) =1 - 2^{-n+1}(m_0 + \abs m) \,,
\end{equation*}
and that is what we wanted to prove.
\end{proof}
\begin{proof}[Proof of Theorem~\ref{thm_main}]
We only have to observe that 
\[
T_{2^{-n}}-C_f = (T_{2^{-n}}-I) -(C_f-I)=\Delta_n-C_{f-\delta} \qedhere
\]
\end{proof}
The obtained result is in a way rather peculiar, as is shown in the following two examples.
\begin{ex}
   The best approximation of the symmetric difference operator $2^{n-1}(T_{2^{-n}}-T_{2^{-n}})$ by a dyadic convolution operator is by the zero operator. This follows easily from the inspection of the diagonal elements (see Proposition \ref{prop:bestappr-dyconvop})
\[
\inprod{T_{2^{-n}}w_k,w_k} - \inprod{T_{-2^{-n}}w_k,w_k}=0 \,,
\]
as
\[
\inprod{T_{-2^{-n}}w_k,w_k}=\inprod{w_k,T_{2^{-n}}w_k}=\inprod{T_{2^{-n}}w_k,w_k} \,.
\]
\end{ex}
\begin{ex}
   The operator $\mJ$ of \emph{anti-differentiation} on \qu\ is 
\[
\mJ f(x) = \int_0^x f(t) \dd t \,.
\] 
Let us determine  the best approximation of the \HS-operator $\mJ$ by a dyadic convolution operator. By Proposition~\ref{prop:bestappr-dyconvop} it is sufficient to compute $\inprod{\mJ w_k, w_k}$. Using  the expansion ~\cite[Eq.(3.6)]{Fi49}, valid for $k \geq 1$,
\[
\mJ w_k = 2^{-n-2} (w_{k'}-\sum_{r=1}^\infty 2^{-r}w_{2^{n+r}+k}) \,,
\]
where  $k\geq1$, $k=2^n+k'$, and $0 \leq k' <2^n$,
we obtain that $\inprod{\mJ w_k,w_k} =0$ for all  $k\geq 1$. On the other hand it is straightforward that $\inprod{\mJ w_0, w_0}=1/2$, so $\gamma_0=1/2$ and $\gamma_k=0$ for all $k \geq 1$. The resulting operator is
\[
D_\gamma f = \frac{1}{2} \int_0^1 f(\lambda) \dd \lambda \,. 
\]
\end{ex}

We end this note with some questions, which might be tractable by the methods used above.

\begin{enumerate}
\item What are the best approximations of $\Delta_n^k$ , $k>1$ by dyadic convolution operators?
\item It is possible to adapt the approach given above and to consider classical differentiation operators on the space of trigonometric polynomials of degree $n$ and their approximation by dyadic convolution operators. Does this change the result of Theorem~\ref{thm_main}?
\item Can a similar result be obtained if the \HS\ norm is replaced by, e.g., the operator norm?
\item What is the generalization to  Vilenkin groups?
\end{enumerate}
\bibliography{walshbib} 
\bibliographystyle{abbrv}
\end{document}

%% file: definitionen1.tex
\providecommand {\norm}[1] {\lVert#1\rVert}

\providecommand {\abs}[1] {\lvert#1\rvert}
\providecommand {\bigabs}[1] {\Bigl\lvert#1\Bigr\rvert}
\providecommand {\inprod}[1]{\langle #1 \rangle}
\providecommand {\set}[1]{\lbrace #1 \rbrace}

\providecommand {\floor}[1]{\lfloor #1 \rfloor}

\providecommand {\inv}[1]{{#1}^{-1}}

\newcommand {\bn} {\ensuremath{\mathbb{N}}}
\newcommand {\bno} {\ensuremath{\mathbb{N}_0}}

\newcommand {\bz} {\ensuremath{\mathbb{Z}}}
\newcommand {\bt} {\ensuremath{\mathbb{T}}}
\newcommand {\bc} {\ensuremath{\mathbb{C}}}

\newcommand {\mD} {\ensuremath{\mathcal{D}}}

\newcommand {\mF} {\ensuremath{\mathcal{F}}}
\newcommand {\mJ} {\ensuremath{\mathcal{J}}}
{\newcommand {\mG} {\ensuremath{\mathcal{G}}}

\newcommand {\mB} {\ensuremath{\mathcal{B}}}

\newcommand {\mS} {\ensuremath{\mathcal{S}}}

\DeclareMathOperator{\dd}{\mathrm{d}}
\DeclareMathOperator{\diag}{Diag}

\newtheorem{prop}{Proposition} 

\newtheorem{thm}[prop]{Theorem}

\newtheorem{lem}[prop]{Lemma}

\theoremstyle{definition}

\theoremstyle{remark}
\newtheorem*{rem}{Remark}
\newtheorem*{rems}{Remarks}
\newtheorem{ex}[prop]{Example}


%% file: introv3.tex
This note originates from  a question of Paul Butzer and Heinz-Joseph Wagner \cite{BuWa75}, posed more than 35 years ago: Is there an ``intuitive'' explanation or interpretation of the \emph{dyadic derivative}, i.e., of the operator $D$ on $L^2([0;1])$, defined by
\begin{equation}
  \label{eq:1}
  D w_k = k w_k \,,
\end{equation}
where $w_k$ is the $k$th \wf\ in Paley enumeration (definitions below)? More precisely, it was stated in ~\cite{BuWa75}:

\emph{``One essential open problem  in  dyadic analysis [...] is an actual interpretation of  the dyadic 
derivative. [..] However,  just  as  the  classical  derivative  may be  associated with  the  slope  of a  tangent to  a  curve,  or with  the  rate  of speed  of an  object,  thus  associated with geometrical 
or  physical notions,  the problem  here  is  to  find an  appropriate  intuitive interpretation of the  dyadic derivative  in  terms of one or more of the modern sciences which 
make use of Walsh analysis.'' }

Despite the generalizations that that the concept of a dyadic derivative has  undergone in the past decades (see, e.g.~\cite{On79,Ze83,BuEnWi86,BuEnWi86b,Zelin93,Sta94}) and the considerable progress in the understanding of dyadic derivatives (a by no means complete list might contain \cite{Qiu11,weisz10, golubov06,simon01,fridli01} and the works referenced in \cite{ScWaSi90})   the question of Butzer and Wagner (see also \cite{En85} for a similar statement) remains without answer -- even if one widens the view to include the generalizations mentioned above.

In the present contribution a generalized dyadic derivative is characterized by its relation to the ordinary differentiation operator. More precisely it is shown that the generalized dyadic derivative can be characterized as the \emph{unique} dyadic convolution operator whose restriction to  the cyclic subgroups $\fdgn$ of $[0;1)$ is the best approximation to the sequence of cyclic difference operators $\Delta_{2^{-n}}$ in the \HS\ norm.

The following properties of this operator are worth to mention:
\begin{itemize}
\item [(a)]  By construction, one  obtains a sequence of best approximating dyadic convolution operators on every subgroup $\fdgn$. The surprising fact is that all of them are restrictions of \emph{one} generalized dyadic derivative.
\item [(b)] Though this operator is not the dyadic derivative invented by Butzer and Wagner it is a closely related operator $D_\gamma$ satisfying
  \begin{equation*}
    \label{eq:2}
    D_\gamma w_{\mG k}=
    \begin{cases} 2\, k\; w_{\mG k}\,,\quad  &k \text{  even}\,,\\
               2 (k +1) w_{\mG k}\,,\quad  &k \text  {  odd}\,,
    \end{cases}
  \end{equation*}
where the $w_{\mG k}$ are the \wf s ordered by increasing number of their sign changes (the so called \emph{sequency ordering}). This might not come as a surprise as ``sequency mimics frequency'' (e.g., the location of the maximum of the Fourier transform of a given \wf\ is related to its sequency index $\mG k$ in a simple manner).
\end{itemize}
Returning to the question of Butzer and Wagner the dyadic derivative ``may be  associated with  the  slope  of a  tangent to  a  curve'' by means of an optimality principle.